\documentclass[a4paper]{amsart}

\usepackage[utf8]{inputenc}
\usepackage{amsthm, amssymb, amsmath, amsfonts}
\usepackage{url}

\newtheorem{thm}{Theorem}[section]
\newtheorem{prop}[thm]{Proposition}
\newtheorem{lem}[thm]{Lemma}

\renewcommand{\le}{\leqslant}
\renewcommand{\ge}{\geqslant}

\newcommand{\E}{\mathbb{E}}
\newcommand{\EE}{\mathbf{E}}

\newcommand{\N}{\mathbb{N}}

\newcommand{\1}{\mathbf{1}}

\renewcommand{\P}{\mathbb{P}}
\newcommand{\PP}{\mathbf{P}}

\newcommand{\eps}{\varepsilon}
\def\d{{\mathrm{d}}}

\newcommand{\bho}{{\beta,h,\omega}}

\newcommand{\F}{\textsc{f}}

\title{On the delocalized phase of the random pinning model}
\author{Jean-Christophe Mourrat}

\address{Ecole polytechnique fédérale de Lausanne, institut de mathématiques, station 8, 1015 Lausanne, Switzerland}

\begin{document}
\begin{abstract}
We consider the model of a directed polymer pinned to a line of i.i.d. random charges, and focus on the interior of the delocalized phase. We first show that in this region, the partition function remains bounded. We then prove that for almost every environment of charges, the probability that the number of contact points in $[0,n]$ exceeds $c \log n$ tends to $0$ as $n$ tends to infinity. Our proofs rely on recent results of \cite{bgh10,ch10}.
\end{abstract}
\maketitle
\section{Introduction}
Let $\tau = (\tau_i)_{i \in \N}$ be a sequence such that $\tau_0 = 0$ and $(\tau_{i+1}-\tau_i)_{i \ge 1}$ are independent and identically distributed random variables with values in $\N^* = \{1,2,\ldots\}$. Let $\PP$ be the distribution of $\tau$, $\EE$ the associated expectation, and $K(n) = \PP[\tau_1 = n]$. We assume that there exists $\alpha \ge 0$ such that
\begin{equation}
\label{regvar}
\frac{\log K(n)}{\log n} \xrightarrow[n \to \infty]{} -(1+\alpha).
\end{equation}
As an example, one can think about the sequence $\tau$ as the sequence of arrival times at $0$ of a one-dimensional simple random walk (and in this case, $\alpha = 1/2$). In a slight abuse of notation, we will look also at the sequence $\tau$ as a set, and write for instance $n \in \tau$ instead of $\exists i : n = \tau_i$. 

Let $\omega = (\omega_k)_{k \in \N}$ be independent and identically distributed random variables. We write $\P$ for the law of $\omega$, and $\E$ for the associated expectation. We will refer to $\omega$ as the \emph{environment}. We assume that the $\omega_k$ are centred random variables, and that they have exponential moments of all order. 
%
%
Let $\beta \ge 0, h \ge 0$, and $n \in \N^*$. We consider the probability measure $\PP_n^\bho$ (expectation $\EE_n^\bho$) which is defined as the following Gibbs transformation of the measure $\PP$~:
$$
\frac{\d \PP_n^\bho}{\d \PP}(\tau) = \frac{1}{Z_n^\bho} \ \exp\left( \sum_{k = 0}^{n-1} (\beta \omega_k - h) \1_{\{k \in \tau\}} \right) \1_{\{n \in \tau\}}.
$$
In the above definition, $\beta$ can be thought of as the inverse temperature, $h$ as the disorder bias, and $Z_n^\bho$ is a normalization constant called the \emph{partition function},
$$
Z_n^\bho = \EE\left[ \exp\left( \sum_{k = 0}^{n-1} (\beta \omega_k - h) \1_{\{k \in \tau\}} \right) \1_{\{n \in \tau\}} \right].
$$
At the exponential scale, the asymptotic behaviour of the partition function is captured by the \emph{free energy} $\F(\beta,h)$ defined as
$$
\F(\beta,h) = \lim_{n \to +\infty} \frac{1}{n} \ \log Z_n^\bho.
$$
Superadditivity of the partition function implies that this limit is well defined almost surely, and that it is deterministic (see for instance \cite[Theorem 4.1]{gg}). Assumption~(\ref{regvar}) implies that $\F(\beta,h) \ge 0$. It is intuitively clear that the free energy can become strictly positive only if the set $\tau \cap [0,n]$ is likely to contain many points under the measure $\PP_n^\bho$. We thus say that we are in the \emph{localized phase} if~$\F(\beta,h) > 0$, and in the \emph{delocalized phase} otherwise. One can show \cite[Theorem~11.3]{holl} that for every $\beta \ge 0$, there exists $h_c(\beta) \ge 0$ such that
$$
\begin{array}{lll}
h < h_c(\beta) & \Rightarrow & \text{localized phase, i.e. } \F(\beta,h) > 0, \\
h \ge h_c(\beta) & \Rightarrow & \text{delocalized phase, i.e. } \F(\beta,h) = 0,
\end{array}
$$
and moreover, the function $\beta \mapsto h_c(\beta)$ is strictly increasing. 
\section{Statement of the main results}
We focus here on the interior of the delocalized phase, that is to say when $h > h_c(\beta)$. Note that, due to the strict monotonicity of the function $h_c(\cdot)$, one sits indeed in the interior of the delocalized phase if one fixes $h = h_c(\beta_0)$ and considers any inverse temperature $\beta < \beta_0$.

By definition, the partition function is known to grow subexponentially in this region. In \cite[Remark p.~417]{bs10}, the authors ask whether the partition function remains bounded there. We answer positively to this question, and can in fact be slightly more precise.
\begin{thm}
\label{Zborne}
Let $\beta \ge 0$ and $h > h_c(\beta)$. For almost every environment, one has
$$
\sum_{n = 1}^{+ \infty} Z_n^\bho < + \infty.
$$
\end{thm}
\noindent \textbf{Remark.}
This result implies that, in the interior of the delocalized phase, the \emph{unconstrained} (or \emph{free}) partition function $Z_{n,f}^\bho$ is also almost surely bounded (in fact, tends to $0$) as $n$ tends to infinity. Indeed, $Z_{n,f}^\bho$ is defined by
$$
Z_{n,f}^\bho = \EE\left[ \exp\left( \sum_{k = 0}^{n-1} (\beta \omega_k - h) \1_{\{k \in \tau\}} \right) \right],
$$
which is equal to
$$
\sum_{n'=n}^{+\infty}\EE\left[ \exp\left( \sum_{k = 0}^{n-1} (\beta \omega_k - h) \1_{\{k \in \tau\}} \right) ; \tau \cap [n,n'] = \{n'\}\right] \le \sum_{n'=n}^{+\infty} Z_{n'}^\bho \xrightarrow[n \to \infty]{\text{a.s.}} 0.
$$

\vspace{10pt}

Our second result concerns the number of points in the set $\tau \cap [0,n]$ under the measure $\PP_n^\bho$. Let us write $E_{n,N}$ for the event that $|\tau \cap [0,n]| > N$ (where we write $|A|$ for the cardinal of a set $A$).

\begin{thm}
\label{contact}
Let $\beta \ge 0$ and $h > h_c(\beta)$. For every $\eps > 0$ and for almost every environment, there exists $N_\eps,C_\eps > 0$ such that for any $N \ge N_\eps$ and any $n$~:
$$
\PP_n^\bho(E_{n,N}) \le \frac{C_\eps}{K(n)} \ e^{-N(h - h_c(\beta) - \eps)}.
$$
In particular, for every constant $c$ such that
$$
c > \frac{1+\alpha}{h-h_c(\beta)}
$$
and for almost every environment, one has
$$
\PP_n^\bho(E_{n,c \log n}) \xrightarrow[n \to \infty]{} 0.
$$
\end{thm}
To our knowledge, results of this kind were known only under the averaged measure $\P\PP_n^\bho$, and with some restrictions on the distribution of $\omega$ due to the use of concentration arguments (see \cite{gt05} or \cite[Section~8.2]{gg}).

\section{Proofs}
In this section, we present the proofs of Theorems~\ref{Zborne} and \ref{contact}. 
Although one might think at first that such an approach cannot be of much help as far as the delocalized phase is concerned, we will rely on recent results obtained in \cite{bgh10,ch10}, where the authors develop a large deviations point of view of the problem. 
Let us define
\begin{equation}
\label{def:F_N}
F_N^\bho = \sum_{0 = l_0 < l_1 < \cdots < l_N} \prod_{i=0}^{N-1} K(l_{i+1}-l_i)e^{(\beta \omega_{l_i} - h)}.
\end{equation}
Our results are based on the following fact, due to \cite{ch10}, that holds both in the delocalized and in the localized phases. 
\begin{lem}
\label{lem:ch10}
For almost every environment, one has
$$
\limsup_{N \to +\infty} \frac{1}{N} \log F_N^\bho = h_c(\beta) - h.
$$
\end{lem}
\begin{proof}[Proof of Lemma~\ref{lem:ch10}]
Although this result is not stated as a proposition in \cite{ch10}, the authors give all the necessary elements to prove it. Indeed, we can start from \cite[(3.11)]{ch10}, which reads
$$
\limsup_{N \to +\infty} \frac{1}{N} \log F_N^\bho = -h + S^{\text{que}}(\beta;1),
$$
where $S^{\text{que}}(\beta;z)$ is defined in \cite[(3.10)]{ch10}. We then learn from \cite[(3.13)]{ch10} that 
$$
h_c(\beta) = S^{\text{que}}(\beta;1^-),
$$
so what remains to see is that 
$$
S^{\text{que}}(\beta;1^-) = S^{\text{que}}(\beta;1).
$$
Clearly, $S^{\text{que}}(\beta;1^-) \le S^{\text{que}}(\beta;1)$. On the other hand, it is shown in step 1 of the proof of \cite[Lemma 3.3]{ch10} that $S^{\text{que}}(\beta;1) \le A(\beta)$, where $A(\beta)$ is defined in \cite[(3.21)]{ch10}. Finally, step 4 shows that $A(\beta) \le S^{\text{que}}(\beta;1^-)$, hence $S^{\text{que}}(\beta;1) \le S^{\text{que}}(\beta;1^-)$, which finishes the proof.
\end{proof}

\begin{proof}[Proof of Theorem~\ref{Zborne}]
The proof is close to \cite[Section~3.2]{ch10}. We can decompose $Z_n$ the following way~:
$$
Z_n^\bho = \sum_{N = 1}^{+\infty} \sum_{0 = l_0 < l_1 < \cdots < l_N = n} \prod_{i=0}^{N-1} K(l_{i+1}-l_i)e^{(\beta \omega_{l_i} - h)}.
$$
An interversion of sums then leads to
$$
\sum_{n = 1}^{+\infty} Z_n^\bho = \sum_{N= 1}^{+ \infty} F_N^\bho,
$$
and Lemma~\ref{lem:ch10} ensures the almost sure convergence of the second series when $h > h_c(\beta)$.
\end{proof}

For an event $A$, let us write $Z_n^\bho(A)$ for the quantity
$$
\EE\left[ \exp\left( \sum_{k = 0}^{n-1} (\beta \omega_k - h) \1_{\{k \in \tau\}} \right) \1_{\{n \in \tau\}} \ ; \ A \right].
$$
In words, $Z_n^\bho(A)$ is a partition function in which one integrates with respect to $\PP$ only on the event $A$. In order to prove Theorem~\ref{contact}, we first give a refined version of Theorem~\ref{Zborne}, which goes as follows. 

\begin{prop}
\label{p:main}
Let $\beta \ge 0$ and $h > h_c(\beta)$. For every $\eps > 0$ and for almost every environment, there exist $N_\eps,C_\eps$ such that for any $N \ge N_\eps$~:
$$
 \sum_{n = 1}^{+\infty} Z_n^\bho(E_{n,N}) \le C_\eps e^{-N(h-h_c(\beta)-\eps)}.
$$
\end{prop}

\begin{proof}
We can assume that $\eps < h - h_c(\beta)$. Note that, for any $n$ and $N_0$,
$$
Z_n^\bho(E_{n,N_0}) = \sum_{N = N_0}^{+\infty} \sum_{0 = l_0 < l_1 < \cdots < l_N = n} \prod_{i=0}^{N-1} K(l_{i+1}-l_i)e^{(\beta \omega_{l_i} - h)}.
$$
By an interversion of sums, we obtain that
$$
\sum_{n = 1}^{+\infty} Z_n^\bho(E_{n,N_0}) = \sum_{N = N_0}^{+\infty} F_N^\bho.
$$
By Lemma~\ref{lem:ch10}, there exists $N_\eps$ such that for every $N \ge N_\eps$,
$$
F_N^\bho \le e^{-N(h - h_c(\beta) - \eps/2)},
$$
and as a consequence, for every $N_0 \ge N_\eps$, one has
$$
\sum_{n = 1}^{+\infty} Z_n^\bho(E_{n,N_0}) \le \sum_{N = N_0}^{+\infty} e^{-N(h - h_c(\beta) - \eps/2)},
$$
which implies the announced claim.
\end{proof}

\begin{proof}[Proof of Theorem~\ref{contact}]
Note that
$$
\PP^\bho(E_{n,N}) = \frac{Z_n^\bho(E_{n,N})}{Z_n^\bho}.
$$
The numerator can be bounded from above using Proposition~\ref{p:main}. For the denominator, one can use the bound
$$
Z_n^\bho \ge K(n) e^{\beta \omega_0-h},
$$
which proves the desired result.
\end{proof}

\end{document}